\documentclass[12pt]{article}
\usepackage{amssymb,amscd}
\usepackage{amsmath, amsthm}
\usepackage{psfrag}
\usepackage{graphicx}
\usepackage{comment}
\usepackage{amsrefs}

\begin{document} \sloppy

\DeclareGraphicsRule{.bmp}{bmp}{}{}
\renewcommand{\refname}{References}

\newcommand{\eps}{\varepsilon}
\newcommand{\Bb}{\mathbb}

\newcommand{\be}{\par\noindent}

\newtheorem{te}{Theorem}[section]
\newtheorem{problem}{Problem}[section]
\newtheorem{lemma}{Lemma}[section]
\newtheorem{pr}{Proposition}[section]
\newtheorem{corollary}{Corollary}[section]
\newtheorem{claim}{Claim}[section]
\theoremstyle{definition}
\newtheorem{df}{Definition}[section]
\newtheorem{example}{Example}[section]
\newtheorem{fact}{Fact}[section]
\renewcommand{\proofname}{\bf Proof:}

\newtheorem{note}{Note}[section]

\newcommand{\g}{\gamma}
\newcommand{\G}{\Gamma}
\newcommand{\p}{\partial}
\renewcommand{\phi}{\varphi}
\renewcommand{\epsilon}{\varepsilon}
\newcommand{\D}{\Delta}
\renewcommand{\d}{\delta}
\newcommand{\de}{\partial}
\renewcommand{\th}{\bigskip\noindent{\bf Theorem }}

\begin{center}
{\Large {\bf Topology of generic holomorphic foliations on Stein
manifolds: structure of leaves and Kupka-Smale property.}}
\end{center}

\begin{center}
Tanya Firsova.

\end{center}
\vspace{0,5 cm}

\begin{abstract}
We study topology of leaves of $1$-dimensional singular
holomorphic foliations of Stein manifolds. We prove that for a
generic foliation all leaves, except for at most countably many
are contractible, the rest are topological cylinders. We show that
a generic foliation is complex Kupka-Smale.
\end{abstract}
\section{Introduction}

Consider a system of differential equations
\begin{equation}
\begin{array}{l}
x_1'=f_1(x_1,\dots, x_n) \\\hdots\\ x_n'=f_n(x_1,\dots, x_n)
\end{array},
\end{equation}

\noindent where $(x_1,\dots, x_n)\in \mathbb C^n$, $f_1,\dots,
f_n\in {\cal O}(\mathbb C^n)$.

The phase space $\mathbb C^n$, outside the singular locus, is foliated by Riemann surfaces.
It is a natural question: what is the topological type of these leaves? For polynomial foliations of
fixed degree this question was asked by Anosov and still remains unsolved. In general, it can be quite
complicated. Consider, for example, a Hamiltonian foliation of $\mathbb C^2$: $H_n=\mbox{const}$,
where $H_n$ is a generic polynomial of degree $n$. All non-singular leaves are Riemann surfaces with
$\frac{(n-1)(n-2)}{2}$ handles and $n$ punctures. There are examples of foliations with dense leaves, having infinitely
generated fundamental groups \cite{Moldavskis}.

So one can restrict the question: what is the topological type of
leaves for a generic foliation?

The genericity here is understood as follows: the space of holomorphic foliations can be naturally
equipped with the (Baire) topology of uniform convergence on nonsingular compacts sets. We
recall the definition of the topology in Appendix \ref{topology}. We call a foliation generic if it belongs to a
residual set -- an intersection of countably many open everywhere dense sets.

In our paper we describe the topological type of leaves for generic foliations on $\mathbb C^n$, and
more generally, on arbitrary Stein manifolds. We prove the following theorem:

\begin{te}\label{main1} For a generic $1$-dimensional singular holomorphic
foliation on a Stein manifold $X$ all leaves, except for at most
countably many, are contractible, the rest are topological
cylinders.
\end{te}

We consider foliations with singular locus of codimension $2$,
i.e. foliations locally determined by holomorphic vector fields
\cite{IlYa}.

Our technique is applicable in a more general setting. In
particular, we establish the analog of Kupka-Smale theorem for
generic foliations on Stein manifolds:

\begin{te}\label{main2} A generic $1$-dimensional singular holomorphic foliation on $X$ is complex Kupka-Smale.
\end{te}

\begin{df} A foliation of a complex manifold is called {\it complex Kupka-Smale} if
\begin{enumerate}
\item all its singular points are complex hyperbolic;
\item all complex cycles are hyperbolic;
\item strongly invariant manifolds of different singular points intersect
transversally;
\item invariant manifolds of complex cycles intersect transversally
with each other and with strongly invariant manifolds of singular
points.
\end{enumerate}
\end{df}

Let cycle $\gamma$ be a phase curve of a real vector field, then
$\gamma$ is a loop on the phase curve of the complexified vector
field. A complex cycle by definition is a free homotopy class of
loops on a leaf of a foliation. Recall that by definition, a real
Kupka-Smale vector field has hyperbolic cycles only. Condition (2)
is a generalization of this property.

We review notions of complex hyperbolicity and invariant manifolds
in the Appendix.

The above definition was suggested by Marc Chaperon in
\cite{Chap}. In this preprint he studies holomorphic
$1$-dimensional singular foliations on Stein manifolds. He shows
that the property (1) holds for generic foliations. He also gives
the proof of the property (3) for generic foliations on $\mathbb
C^n$ and states the result for generic foliations on Stein
manifolds. Our technique also allows us to prove transversality
results for strongly invariant manifolds of the same singular
point:
\begin{te}\label{te:homoclinic_int} For a generic 1-dimensional singular holomorphic
foliation: \begin{enumerate}
\item all singular points are complex hyperbolic.
\item Let $a_1$ be a complex hyperbolic singular point of the foliation. Let $M_1$ and $M_2$
be strongly invariant manifolds of the point $a_1$, such that
$M^{loc}_1\cap M^{loc}_2=a_1$. Then $M_1$ and $M_2$ intersect
transversally everywhere.
\end{enumerate}
\end{te}

Theorems \ref{main1}, \ref{main2} for foliations of $\mathbb C^2$ are proved in \cite{TF}. Golenishcheva-Kutuzova
\cite{GK} showed that for a generic foliation countable many cylinders do exist.
We expect that for a generic singular holomorphic 1-dimensional foliation of a Stein manifold
there are countably many cylinders.

The conformal type of leaves of a generic polynomial foliation of fixed degree
was described by Candel, Gomez-Mont \cite{CGM}. The result was later improved by Lins Neto
\cite{LN94}, and Glutsyuk\cite{G94}:
\begin{te}\cite{G94}, \cite{LN94} Any leaf of a generic polynomial foliation of degree $n$
is hyperbolic.
\end{te}
We expect that the same answer is true for generic foliations of Stein manifolds. The technique from \cite{CGM}, \cite{LN94},
\cite{G94} can be adjusted to attack the problem. See the paper
\cite{Il08} for a vast discussion of open problems.

Greg Buzzard studied similar genericity questions for polynomial
automorphisms of $\mathbb C^n$. He proved that a generic
polynomial automorphism of $\mathbb C^n$ is Kupka-Smale
\cite{Buzzard}.

\subsection{Outline of the article.}

We establish generic properties of foliations by constructing
perturbations that eliminate degeneracies. There are at most
countably many isolated cycles. (This lemma is proved in \cite{LP}
for foliations of $\mathbb C^2$. We included the proof for
arbitrary Stein manifolds in Section \ref{LP_lemma} to explain our
strategy of simultaneous elimination of degeneracies.) Therefore,
once all nonisolated cycles are removed, all leaves, except for
countably many, are contractible.

To prove that the rest have fundamental group $\mathbb Z$, one
needs to eliminate all degeneracies from the following list:
\begin{enumerate}
\item two cycles that belong to the same leaf of the foliation and
are not multiples of the same cycle in the homology group of the
leaf;
\item saddle connections;
\item cycles on a separatrix that are not multiples of the cycle
around the critical point.
\end{enumerate}

Recall that a separatrix is a leaf that can be holomorphically
extended into a singular point and a saddle connection is a common
separatrix of two singular points.

In the smooth category one can remove a degeneracy of the
foliation locally. Say, one can destroy a homoclinic loop by
changing the foliation only in a flow-box around a point on the
loop.

In the holomorphic category, a priori, one cannot perturb a
foliation in a flow-box without changing the foliation globally.
Our strategy to remove degeneracies in the holomorphic category is
the following:

In Section \ref{local_degeneracy} we construct a family of
foliations, that eliminate degeneracy, in a neighborhood of a
degenerate object, rather than in a flow-box around a point. A
non-isolated cycle, a non-trivial pair of cycles are examples of
degenerate objects. We give a complete list of degenerate objects
in Section \ref{local_degeneracy}. All degenerate objects we
consider are curves. Our technique allows us to construct an
appropriate family only if a degenerate object is holomorphically
convex. We expect though that it should be possible to carry out
for any degenerate object.

In \cite{TF} our approach to construct a family of local
foliations in a neighborhood of a degenerate object was to control
the derivative of the holonomy map along the leaf with respect to
a perturbation. This approach can not be adapted to remove a
non-transversal intersection of strongly invariant manifolds. One
cannot choose a leaf-wise path, that connects singular points with
a point of non-transversal intersection. Therefore, one cannot
control the intersection of invariant manifolds.

In this paper we use a different approach, a more geometric one.
First, we reglue the neighborhood (Subsection \ref{regluing}).
Then we project the obtained manifold, together with a new
foliation, to the original one. We use Theorem \cite{Siu}, that
states that a Stein manifold has a Stein neighborhood, to
construct the projection.

We give a review of results on the holomorphic hulls of
collections of curves in Section \ref{global}. We apply them to
give geometric conditions for degenerate object to be
holomorphically convex. We also review the relevant results from
the Approximation Theory on Stein manifolds and apply them to pass
from a local family of foliations in a neighborhood of a
degenerate object to a global one.

When we remove a degenerate object, e.g. a complex cycle, we do
not control the foliation outside a neighborhood of the degenerate
object. Therefore, it might happen that eliminating one degenerate
object we create many other in different places. We solve this
problem as follows: We find a countable number of places where
degenerate objects can be located. For each such location we prove
that the complement to the set of foliations, which have the
degenerate object at this particular location, is open and
everywhere dense. Then we intersect these sets and get a residual
set of foliations without holomorphically convex degenerate
objects. We show that if a foliation has a degenerate object, then
it has a holomorphically convex degenerate object. Therefore, the
residual set constructed does not have degenerate objects. We
describe this strategy in detail in Section \ref{simultaneous}.
This strategy was previously used in \cite{TF} and \cite{GKK}.

We give background information on Stein manifolds in the Appendix to make the paper accessible to the specialists, working
in Dynamical Systems. There is also background information on
holomorphic hulls and complex foliations in the Appendix. We also review facts
on multiplicity of analytic sets, .

\subsection{Acknowledgements}
The author is grateful to Yulij Ilyashenko for the statement of
the problem, numerous discussions and useful suggestions. We are
also grateful to Igors Gorbovickis for proofreading an earlier version of the manuscript.

\section{Local removal of degenerate objects.}\label{local_degeneracy}

\subsection{List of degenerate objects.}

As we pointed out in the introduction one can not eliminate a
homoclinic saddle connection by changing the foliation only
locally in a flow-box. Rather than that one needs to perturb the
foliation in the neighborhood of the separatrix loop. This leads
us to considering degenerate objects.

Below we list degenerate objects. One can check that if a
foliation does not have degenerate objects of type 1-5, then it
satisfies Theorem \ref{main1}. If all singular points of a
foliation are complex hyperbolic and it does not have degenerate objects
of types $1-6$, $8-9$, then it is Kupka-Smale. If a all singular points of a foliation
are complex hyperbolic and a foliation does not have degenerate objects of type $7$, then
it satisfies Theorem \ref{te:homoclinic_int}.

\begin{df} We say that $\gamma$ is a degenerate object of a
foliation ${\cal F}$ if $\gamma$ is
\begin{enumerate}

\item A non-trivial loop on a leaf $L$ of $\cal
F$, which is a representative of a non-hyperbolic cycle.

\item A union of loops $\gamma_1,\gamma_2$ that belong to the same leaf $L$ of $\cal
F$. We assume $\gamma_1$ and $\gamma_2$ are not multiples of the same cycle. Moreover,
$\gamma_1$, $\gamma_2$ are hyperbolic. (See Fig. \ref{fig:pair_cycles}.)

\begin{figure}[h!]
\centering
\psfrag{gamma1}{$\gamma_1$}\psfrag{gamma2}{$\gamma_2$}
\includegraphics[height=4.5cm]{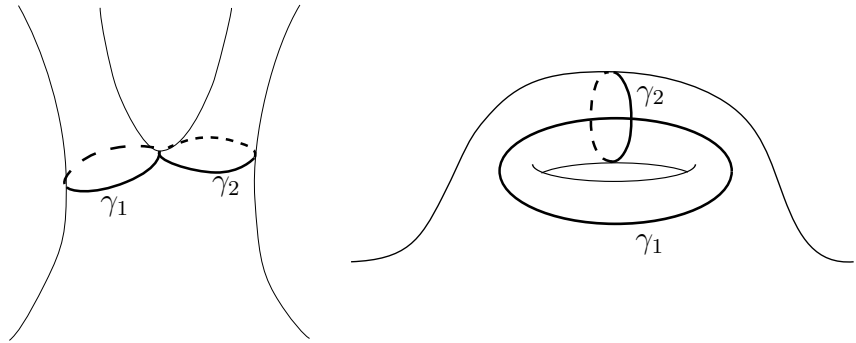}
\caption{A pair of cycles}
\label{fig:pair_cycles}
\end{figure}

\item A path on a saddle connection, that connects two different hyperbolic singular
points $a_1$ and $a_2$. (See Fig. \ref{fig:saddle_connection}).
\item A loop on a homoclinic saddle connection $S$ (See Fig.\ref{fig:saddle_connection}):
\begin{itemize}
\item $a$ is a hyperbolic singular point;
\item $S_1$, $S_2$ are local separatrices of the singular point
$a$; $S_1\neq S_2$; $S_1,S_2\subset S$;
\item $\gamma\subset S$ passes through the singular point $a$, starts at $S_1$, ends along $S_2$.
\end{itemize}
\begin{figure}[h!]
\centering
\psfrag{gamma}{$\gamma$}\psfrag{a1}{$a_1$}\psfrag{a2}{$a_2$}\psfrag{a}{$a$}
\includegraphics[height=4cm]{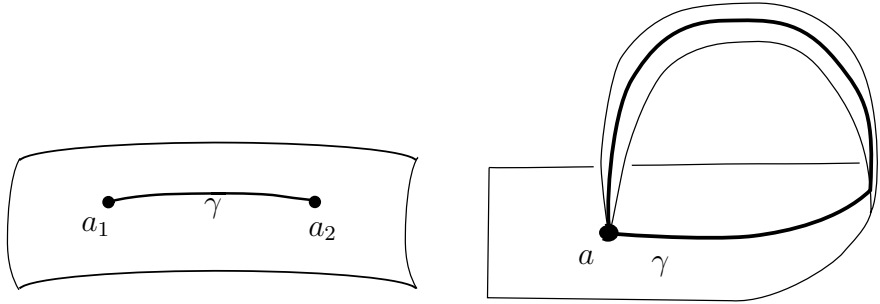}
\caption{A path on a saddle connection and a loop on a homoclinic
saddle connection}\label{fig:saddle_connection}
\end{figure}

\item A non-trivial loop $\gamma$ on a separatrix that passes through a singular point
$a$ (See Fig. \ref{fig:cycle_separatrix}).
\begin{figure}[h!]
\centering
\psfrag{a}{$a$}\psfrag{gamma}{$\gamma$}
\includegraphics[height=3cm]{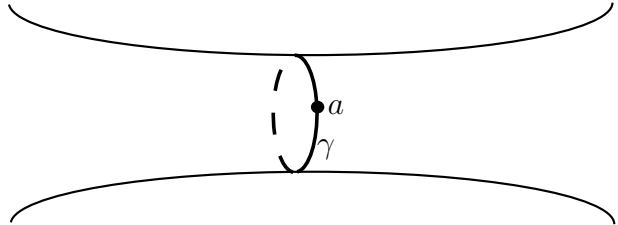} \caption{A loop on a separatrix.}\label{fig:cycle_separatrix}
\end{figure}
\item A union of paths $\gamma_1$ and $\gamma_2$ (See Fig. \ref{fig:st_in_manifolds}):
\begin{itemize}
\item $a_1$, $a_2$ are hyperbolic singular points of the foliation $\cal F$;
\item $M_1$ and $M_2$ are strongly invariant manifolds of $a_1$ and
$a_2$ correspondingly;
\item $p$ is a point of non-transversal intersection of $M_1$ and
$M_2$;
\item $\gamma_1\subset M_1$ and $\gamma_2\subset M_2$ are paths that connect $a_1$ and
$a_2$ with the point $p$;
\item $(\gamma_1\cup \gamma_2)\backslash(M_1^{loc}\cup M_2^{loc})\subset
L$, where $L$ is a leaf of foliation $\cal F$.
\end{itemize}
\begin{figure}[h!]\centering
\psfrag{a1}{$a_1$} \psfrag{a2}{$a_2$}\psfrag{p}{$p$}
\psfrag{M1}{$M_1$}\psfrag{M2}{$M_2$}\psfrag{gamma1}{$\gamma_1$}\psfrag{gamma2}{$\gamma_2$}
\includegraphics[height=7cm]{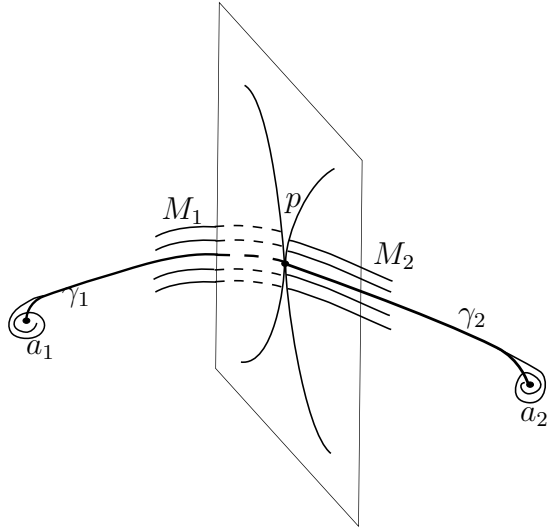}
\caption{A non-transversal intersection of strongly invariant
manifolds. The leaf $L$ on the picture is not
a separatrix, it spirals around singular points $a_1$ and $a_2$}
\label{fig:st_in_manifolds}
\end{figure}
\item A loop $\gamma_1\cup \gamma_2$ (See Fig. \ref{fig:hom_st_in_man}):
\begin{itemize}
\item $a$ is hyperbolic singular point of the foliation $\cal F$;
\item $M_1$ and $M_2$ are strongly invariant manifold of the point
$a$;
\item $M_1^{loc}\cap M_2^{loc}=a$;
\item paths $\gamma_1\subset M_1$, $\gamma_2\subset M_2$ connect
$a$ with $p$;
\item $(\gamma_1\cup \gamma_2)\backslash(M_1^{loc}\cup M_2^{loc})\cup
L$, where $L$ is a leaf of foliation $\cal F$.
\end{itemize}
\begin{figure}[h!]\centering
\psfrag{a}{$a$}
\psfrag{M1}{$M_1$}\psfrag{M2}{$M_2$}\psfrag{gamma1}{$\gamma_1$}\psfrag{gamma2}{$\gamma_2$}\psfrag{p}{$p$}
\includegraphics[height=7cm]{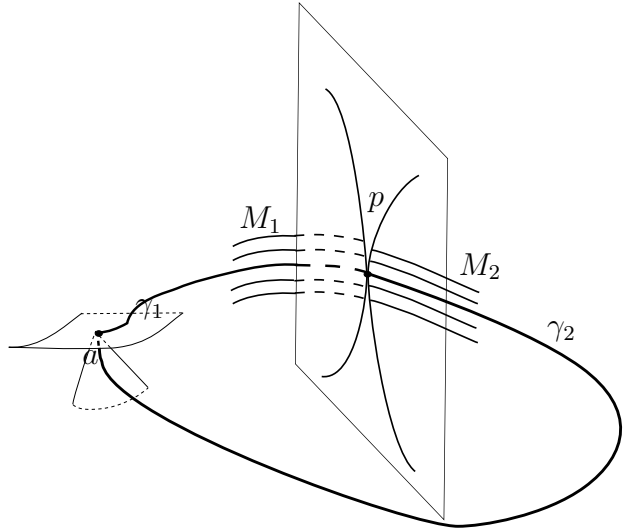}
\caption{A homoclinic non-transversal intersection of strongly
invariant manifolds.}\label{fig:hom_st_in_man}
\end{figure}
\item A union $\gamma_1\cup \gamma_2\cup \gamma_3\cup
\gamma_4$:
\begin{itemize}
\item $\gamma_1$, $\gamma_2$ are hyperbolic loops on leaves of $\cal F$;
\item $M_1, M_2$ are invariant manifolds of $\gamma_1$, $\gamma_2$
correspondingly;
\item $\gamma_3\subset M_1$, $\gamma_4\subset M_2$ are paths that
connect points on $\gamma_1$,$\gamma_2$ with a point of
non-transversal intersection of $M_1,$ $M_2$. \item $(\gamma_3\cup
\gamma_4) \backslash \left(M_1^{loc}\cup M_2^{loc}\right)\subset
L$, where $L$ is a leaf of $\cal F$.
\end{itemize}
\item A union $\gamma=\gamma_1\cup \gamma_2\cup \gamma_3$:
\begin{itemize}
\item $\gamma_1$ is a hyperbolic loop on a leaf;
\item $M_1$ is an invariant manifold of $\gamma_1$;
\item $a$ is a hyperbolic singular point;
\item $M_2$ is a strongly invariant manifold of $a$;
\item $\gamma_2\subset M_1$, $\gamma_3\subset M_2$ are paths on
invariant manifolds that connect a point on $\gamma_1$ and the
point $a$ correspondingly with the the point of non-transversal
intersection of $M_1$ and $M_2$
\item $(\gamma_2\cup
\gamma_3)\backslash\left(M_1^{loc}\cup M_2^{loc}\right)\subset L$,
where $L$ is a leaf of the foliation $\cal F$.
\end{itemize}
\end{enumerate}
\end{df}

\subsection{Local Removal Lemma.}
In this section we find a neighborhood of a degenerate object and
a family of holomorphic foliations in this neighborhood that
eliminate the degenerate object in the neighborhood.

Our technique allows us to do that only if a degenerate object is
holomorphically convex. We expect, though, it should be possible
to carry out for any collection of smooth enough curve.

Let $U$ be a neighborhood of the degenerate object. First, we
allow not only the foliation, but the neighborhood itself to
change with the parameter $\lambda$. We get a family of foliations
${\cal F}_{\lambda}$ on manifolds $U_{\lambda}$. Then we find the
way to 'project'\ $U_{\lambda}$ to some neighborhood of the
degenerate object. Thus, we produce a family of foliations in the
neighborhood of the degenerate object that breaks it.

The following lemma summarizes the results of the following two
subsections.

Let $\gamma$ be a union of curves on a Stein manifold $X$, endowed
with a foliation ${\cal F}_0$. Assume it is holomorphically
convex. Fix a point $p\in \gamma$, assume that $p\not \in
\Sigma({\cal F})$. Assume that in a neighborhood of $p$ the curve
$\gamma$ belongs to a leaf $L$ of ${\cal F}_{0}$. Let
$\alpha\subset \gamma$ be a small arc, a neighborhood of $p$ on
$\gamma$. One can fix coordinates $(z_1,\dots, z_{n-1},t)$ in a
neighborhood of the point $p$, so that $t$ is a coordinate along
the foliation. Consider the flow-box $\Pi=\{(z,t):\ |z|<1, t\in
U(\alpha)\}$, where $U_{\alpha}$ is a neighborhood of an arc
$\gamma$ on the leaf $L$. Take a pair of points $q_1,q_2\in
\gamma\backslash \alpha$, that lie on different sides of $\alpha$
and belong to the flow-box $\Pi$. Let $T_1, T_2$ be transversal
sections to ${\cal F}_0$ that pass through $q_1,q_2$. Functions
$(z_1,\dots, z_{n-1})$ work as coordinates on $T_1, T_2$.

\begin{figure}[h!]\centering
\psfrag{U}{$U$}\psfrag{gamma}{$\gamma$}\psfrag{p}{$p$}
\psfrag{q1}{$q_1$}\psfrag{q2}{$q_2$}\psfrag{T1}{$T_1$}\psfrag{T2}{$T_2$}
\includegraphics[height=4cm]{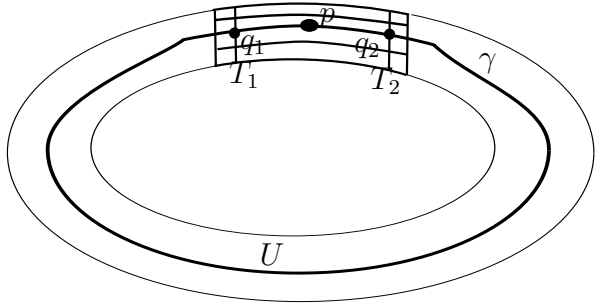}
\caption{$\gamma$ together with its neighborhood}
\end{figure}

Let $\Phi_{\lambda}$ be a holomorphic on $\lambda$ family of germs
of biholomorphisms
$$\Phi_{\lambda}:\left(\mathbb
C^{n-1},0\right) \to \left(\mathbb C^{n-1},0\right),\quad
\Phi_{0}=Id.$$

\begin{lemma}\label{new_foliation}  \noindent There exist a neighborhood $\tilde{U}$ of
$\gamma$, that retracts to $\gamma$, such that $\Pi\cap
\tilde{U}=\{(z,t):\ |z|<\epsilon, t\in U(\alpha)\}$, and a family
of foliations ${\cal F}_{\lambda}$ on $\tilde{U}$ that depends
holomorphically on $\lambda$ satisfying the following conditions:
\begin{enumerate}
\item  in $\tilde{U}\backslash \Pi$, ${\cal F}_{\lambda}$ is biholomorphic
to ${\cal F}_0$.  More precisely, there exists a holomorphic on
$\lambda$ family of maps $\pi_{\lambda}: \left(\tilde{U}\backslash
\Pi\right)\to X,$ which are biholomorphisms to their images, such
that $\pi_{\lambda}$ maps the leaves of ${\cal F}_0$ to the leaves
of ${\cal F}_{\lambda}$, $\pi_0=Id$;
\item The holonomy map inside the flow-box along the foliation ${\cal F}_{\lambda}$
between $T_1$ and $T_2$ is biholomorphically conjugate to
$\Phi_{\lambda}$, more precisely, in coordinates $(z_1,\dots,
z_{n-1})$ on $T_1,T_2$ it is
$\left(\pi^{z}_{\lambda}\right)^{-1}\circ \Phi_{\lambda}\circ
\pi^z_{\lambda}$, where $\pi^z_{\lambda}$ and
$\left(\pi^z_{\lambda}\right)^{-1}$ are first $(n-1)$ coordinates
of $\pi_{\lambda}$ and $\pi^{-1}_{\lambda}$ correspondingly.
\end{enumerate}
\end{lemma}

This lemma mimics the smooth case, where one can perturb the
foliation only in the flow-box. In the holomorphic case this is
not possible. Therefore, we need to adjust everything by the map
$\pi_{\lambda}$.

\subsection{Regluing.}\label{regluing}

We weaken the restriction on curve $\gamma$ for this section. We
do not assume it to be holomorphically convex.

We start by constructing manifolds $U_{\lambda}$. They are
obtained by regluing $U$ in a flow-box around a point $p$. First,
we describe the procedure informally and point out the technical
difficulties that arise. Then we repeat the description paying
attention to the technical difficulties.

We take a neighborhood $U$ that can be retracted to $\gamma$. Let
$\hat{U}$ be a complex manifold obtained form $U$ by doubling the preimage under retraction of a small
arc, containing $p$. One can assume that the preimage of this small arc is a flow-box. So $\hat{U}$
comes with the natural projection $\hat{U}\to U,$ which is one-to-one everywhere except for the two
flow-boxes around the preimages of $p$, which are glued together by the identity map.
$U_{\lambda}$ is obtained from $\hat{U}$ by gluing the points in
the flow-boxes by using the map $(\Phi_{\lambda}, \mbox{Id})$. The
problem is that $(\Phi_{\lambda}, \mbox{Id})$ is not an
isomorphism from the flow-box to itself. Thus, extra caution is
needed to make $U_{\lambda}$ Hausdorff. In the rest of the section
we describe these precautions.

First, we choose a bigger neighborhood $W$ that can be retracted
to $\gamma$. Let $\rho$ denote the retraction. Let $\hat{W}$ be a
connected complex manifold that projects one-to-one on
$U\backslash \rho^{-1}(\alpha)$ and two-to-one on
$\rho^{-1}(\alpha)$. Let $\pi_1^{-1},$ $\pi_2^{-1}$ be the the two inverses of the projection $\hat{W}\to W$,
restricted to the preimage of $\rho^{-1}(a)\subset W$.

Let $V$ denote a flow-box around a point $p$ in $W$. We assume
$V\subset \rho^{-1}(\alpha)$.  We take $V$ small enough so that
$(\Phi_{\lambda}, \mbox{Id})$ is a well-defined map on $V$ and is
a biholomorphism to its image. Let $V_1=\pi_1^{-1}(V)$,
$V_2=\pi_2^{-1}(V)$.

Let $T_c\subset W$ be a a tube of points that are at distance $c$
from the preimage of $\gamma\backslash\alpha$. Let
$\hat{T}_c\subset \hat{W}$ be a tube of points that are at the
distance $c$ from the preimage of $\gamma\backslash\alpha$. Take
$c$ small enough.

Take $U=T_c\cup V$, $\hat{U}=V_1\cup V_2\cup \hat{T}_c$. Note that
$U$ can obtained from $\hat{U}$ by gluing the points from $V_1$
and $V_2$ that project to the same point in $W$.

Let $V_2^{\lambda}=\pi_2^{-1}\left((\Phi_{\lambda}, Id)(V)\right)$

Let $\hat{U_\lambda}=V_1\cup {\hat T}_c\cup V_{2}^{\lambda}.$
$U_{\lambda}$ is a space obtained from $\hat{U}_{\lambda}$ by
gluing $V_1$ and $V_2^{\lambda}$ by the map $(\Phi_{\lambda},
\mbox{Id})$. The space $U_{\lambda}$ inherits complex structure.
If one takes $c$ and $\lambda$ small enough, then it is also
Hausdorff.

We also consider the total space of reglued manifolds:
$$ \hat{\cal U}=\{(u, \lambda)\in \hat{W}\times \Lambda | \ u\in
V_1\cup T_c\cup V_2^{\lambda},\lambda\in \Lambda \} $$
$$ {\cal U}=\hat{\cal U}/ \sim, (u,\lambda)\sim
\left((\Phi_{\lambda}, \mbox{Id})(u), \lambda\right),\
\mbox{where}\ u\in V_1, \lambda\in \Lambda$$

\subsection{Projection. Siu's Theorem.}

In this subsection we prove that for small enough $\lambda$ one can take a small neighborhood
of $\gamma$ in $U_{\lambda}$ and project it biholomorphically to a neighborhood of $\gamma$ in $U$.

Assume $\gamma$ is holomorphically convex. One can choose a
neighborhood $U_1$ of $\gamma,$ $U_1\subset U$, such that $U_1$ is
an analytic polyhedron, therefore, a Stein manifold (\cite{GR}).

By the theorem, formulated below there is a Stein neighborhood $\tilde{\cal U}$ of
$U_1$ in $\cal U$.
\begin{te}(\cite{Siu}) Suppose $X$ is a complex space and $A$ is a
subvariety of $X$. If $A$ is Stein, then there exists an open
neighborhood $\Omega$ of $A$ in $X$ such that $\Omega$ is Stein.
\end{te}

Fix an embedding of $\tilde{\cal U}$ into $C^N$. We will need the following lemma:
\begin{lemma}\label{lem:projection}
There exists a linear $(N-n)$-subspace $\alpha \subset \mathbb
C^N$ such that the affine subspaces $\alpha_x\subset \mathbb C^N$
parallel to $\alpha$ passing through points $x\in \gamma$ are:
\begin{itemize}
\item[a)] transverse to $U;$
\item[b)] pass through only one point on $\gamma.$
\end{itemize}
\end{lemma}

\begin{proof}
The set of all $(N-n)$-subspaces of $\mathbb C^N$ is
$n(N-n)$-dimensional complex manifold $Gr(N-n,N).$

Elements of $Gr(N-n,N)$ that are not transverse to a given
subspace of complementary dimension form a codimension $1$ complex
(singular) subvariety. Path $\gamma$ is $1$-dimensional real
manifold. Therefore, subspaces that do not satisfy (a) form a
subvariety of $Gr(N-n,N)$ of real codimension $1$.

A couple of points on $\gamma$ form a real $2$-dimensional
manifold. Linear subspace parallel to those that pass through two
given points in $\mathbb C^N$ form $(n(N-n-1))$-dimensional
manifold. Therefore, subspaces that do not satisfy (b) form a
submanifold of $Gr(N-n,N)$ of real codimension $2(n-1).$

Since $n\geq 2,$ a $(N-n)$-subspace $\alpha$, that satisfies
conditions (a) and (b), exists.
\end{proof}
\begin{proof}[\bf Proof of Lemma \ref{new_foliation}:]
Take $\alpha$ that satisfies Lemma \ref{lem:projection}. Let $\tilde{\pi}_{\lambda}$
be a projection along $\alpha$ from a neighborhood $\tilde{U}$ of $\gamma$ in $U$ to
$U_{\lambda}$, given by Lemma \ref{lem:projection}. One can take $\tilde{U}$ be small enough, so that
$\pi(\lambda):\tilde{U}\to U_{\lambda}$ is a biholomorphism to its image for all small $\lambda$.
Let $i_{\lambda}:U\backslash\rho^{_1}(\alpha)\to U_{\lambda}$ be an identity map.
It is easy to see that $\pi_{\lambda}=\tilde{\pi}^{-1}_{\lambda}\circ i_{\lambda}$ is a required map.
\end{proof}

\subsection{Removal of a holomorphically convex degenerate object.}

A degenerate object is removed by a small perturbation if, roughly
speaking, in some neighborhood of an object, there are no
degenerate objects of the same kind for perturbed foliations.

Let $\gamma$ be a degenerate object of a foliation ${\cal F}_0$ on
a manifold $X$.

We say that ${\cal F}_{\lambda}$ is a local holomorphic family for
$\gamma$ if there exists a neighborhood $U$ of $\gamma$, such that
${\cal F}_{\lambda}$ are well-defined in $U$ for all $\lambda\in
\Lambda$, $0\in \Lambda$; and ${\cal F}_{\lambda}$ depend
holomorphically on $\lambda$.

\begin{te}\label{th:removal_deg_ob} Let $\gamma$ be a holomorphically convex degenerate
object of a foliation ${\cal F}_0$. Then there exists a local
holomorphic family of foliations ${\cal F}_{\lambda}$ that removes
$\gamma$.
\end{te}

In the following subsection we rigorously define what it means
that a degenerate object is removed in a local holomorphic family
of foliations. We also prove Theorem \ref{th:removal_deg_ob} for
different types of degenerate objects.

\subsection{Removal of a non-hyperbolic cycle.}\label{non-hyperbolic_cycle}

\begin{df} Let $\gamma$ be a non-hyperbolic cycle of a foliation ${\cal F}_0$.
We say that it is removed in a local holomorphic family of foliations
${\cal F}_{\lambda}$ if
\begin{enumerate}
\item there is a transversal section $T$ at a point $p\in \gamma$ to the foliation ${\cal F}_0$ such that holonomy maps along
$\gamma$ for the foliation ${\cal F}_{\lambda}$,
$\Delta^{\lambda}_{\gamma}:\left(D_r\right)\to T$ are well-defined
for $\lambda\in \Lambda$, where $D_r\subset T$ is a disk of radius
$r$ with the center in a point $p$;
\item for all $\lambda\in \Lambda\backslash R$, $\Delta^{\lambda}_{\gamma}$ has a unique fixed point on
$D_r$, where $R$ is a (possibly empty) one dimensional
real-analytic set. Moreover, this fixed point is hyperbolic.
\end{enumerate}
\end{df}

\begin{proof}[\bf Proof of Theorem \ref{th:removal_deg_ob} for type 1:]

Take a point $p\in \gamma$ and a transversal section $T$ to $\cal
F$, $p\in T$. Let $\Delta_{\gamma}:(T,p) \to (T,p)$ be the
corresponding holonomy map. The cycle $\gamma$ is hyperbolic by
the definition if and only if all the eigenvalues of
$\Delta_{\gamma}$ lie not on the unit circle.

First, we provide a specific perturbation of $\Delta_{\gamma}$
that has hyperbolic fixed points only.

The following lemma is the standard fact:

\begin{lemma}\label{perturbation} There exists a diagonal $n\times n$ matrix $D$
and $a\in \mathbb C^n$ such that the map
$\Delta_{\gamma}(z)+\lambda(Dz+a)$ is well-defined and has
hyperbolic fixed points only for all $\lambda\in V\backslash R,$
where $V$ is a neighborhood of $0$, $R$ is a (possibly empty) $1$-dimensional real-analytic set,
$0\in R$.
\end{lemma}

Take $a,D$ such that Lemma \ref{perturbation} is satisfied.

Apply Lemma \ref{new_foliation} to the cycle $\gamma$, the point
$p$ and the family of biholomorphisms
$\Phi_{\lambda}=Id+\lambda(Dz+a)$. The map
$\Delta_{\gamma}^{\lambda}=\pi^{-1}_{\lambda}\circ
\left(\Delta_{\gamma}+\lambda(\Delta z+a)\right)\circ
\pi_{\lambda}$ is the holonomy map along $\gamma$ for the
foliation ${\cal F}_{\lambda}$. For all $\lambda$ outside a
(possibly empty) one-dimensional real-analytic set $R$ the map
$\Delta_{\gamma}^{\lambda}$ has hyperbolic fixed points only on
$T$.
\end{proof}

\subsection{Splitting cycles to different leaves}\label{two_cycles}

Let $\gamma=\gamma_1\cup \gamma_2$ be a degenerate object of type
2.

\begin{df} We say that $\gamma$ is removed in a holomorphic family
of foliations ${\cal F}_{\lambda}$, $\lambda\in \Lambda$ if
\begin{enumerate}
\item there is a transversal section $T$ at a point $p\in \gamma_1\cap \gamma_2$ to
the foliation ${\cal F}_{0}$ such that
holonomy maps
$\Delta^{\lambda}_{\gamma_1},\Delta^{\lambda}_{\gamma_2}: D_r\to
T$ are well-defined for all $\lambda\in \Lambda$, where
$D_r\subset T$ is a disk of radius $r$;
\item $\Delta^{\lambda}_{\gamma_1}$ and $\Delta^{\lambda}_{\gamma_2}$ do not
have a common fixed point on $D_r$ for $\lambda\neq 0$.
\end{enumerate}
\end{df}

\noindent Thus, the degenerate object is removed if $\gamma_1$ and
$\gamma_2$ split to leaves, that are different at least in $U$.

\begin{proof}[\bf Proof of Theorem \ref{th:removal_deg_ob} for type 2:]

Let $q\in \g_1\backslash\g_2$. Assume, it is not a point of
self-intersection of $\gamma_1$. Apply Lemma \ref{new_foliation}
to the curve $\gamma$, the point $q$, and the family of
biholomorphisms $\Phi_{\lambda}=z+\lambda$. Then
$\pi^{-1}_{\lambda}\circ\Delta_{\gamma_2}\circ \pi_{\lambda}$ is a
holonomy map along $\gamma_2$. Let $T_1$ be a transversal section
to the foliation ${\cal F}_0$ in a point $q$. The holonomy map
along $\gamma_1$ for the foliation ${\cal F}_0$ can be written as
a composition $\Delta_{\gamma_1}=\Delta_2\circ\Delta_1$, where
$\Delta_1$ is a holonomy map from transversal section $T$ to
$T_1$, $\Delta_2$ is a holonomy map from $T_1$ to $T$. Then the
holonomy map along $\gamma_1$ for the foliation ${\cal
F}_{\lambda}$ is $\pi^{-1}_{\lambda}\circ \Delta_2\circ
\Phi_{\lambda}\circ \Delta_1\circ \pi_{\lambda}$.

$\pi^{-1}_{\lambda}(p)$ is an isolated fixed point for the
holonomy map along $\gamma_2$ and is not a fixed point for the
holonomy map along $\gamma_1$. Thus, cycles split to different
leaves.
\end{proof}

\subsection{Removal of non-transversal intersections of invariant manifolds and saddle connections.}

Let $\gamma$ be a degenerate object of the foliation ${\cal
F}_{0}$ of types $3-9$. Let ${\cal F}_{\lambda}$ be a family of
local foliations for $\gamma$.

In the sequel the words ``invariant manifold'' stand for a strongly invariant manifold,
or separatrix of a singular point; or stable, unstable manifolds of a complex cycle. These
objects persist under the perturbation, and depend holomorphically on a foliation.
The local strongly invariant manifolds and separatrices of $a_i$
and local stable/instable manifolds of $\gamma_i$ persist under
the perturbation and depend holomorphically on $\lambda$.

For each degenerate object of type $3-9$, there are two invariant manifolds that meet nontransversally. Saddle
connections are examples of non-transversal intersection. We denote the corresponding local invariant
manifolds by $M_1^{loc}$ and $M_2^{loc}$ for the foliation ${\cal F}_0$; $M_1^{loc}(\lambda)$ and
$M_2^{loc}(\lambda)$ for the foliation ${\cal F}_{\lambda}$. Note that for the degeneracy of type $5$,
$M_1^{loc}=M_2^{loc}$.

For degenerate objects of types $3-5$, we can take $p$ to be any point in $\gamma\backslash
\left(M_1^{loc}\cup M_2^{loc}\right)$.

Notice that for all degenerate objects of type $3-9$, $\gamma\backslash\left(M_1^{loc}\cup M_2^{loc}
\right)\subset L$, where $L$ is a leaf of foliation ${\cal F}_0$. Therefore, holomorphic extensions $M_1(\lambda)$
and $M_2(\lambda)$ of $M_1^{loc}(\lambda)$ and $M_2^{loc}(\lambda)$ along $\gamma$ are well-defined.
\begin{df} We say that $\gamma$ can be eliminated in a holomorphic
family of foliations ${\cal F}_{\lambda}$ if there exists a
transversal section $T$ to the foliation ${\cal F}_0$, $p\in T$,
so that $M_1(\lambda)$ and $M_2(\lambda)$ intersect transversally
on $T$.
\end{df}

\begin{note} Note that if $M_1^{loc}$ and $M_2^{loc}$ are
separatrices, then the holomorphic family eliminates the saddle
connection.
\end{note}

\begin{proof}[\bf Proof of Theorem \ref{th:removal_deg_ob} for types 3-9]
One can assume that in a neighborhood of a point $p$ $M_1$ and
$M_2$ are biholomorphically equivalent to $m_1\times D$,
$m_2\times D,$ where  $m_1=M_1\cap D_1,$ $m_2=M_2\cap D_1$, $D$ is
a neighborhood of $p$ on the leaf $L$; $D_1$ is a neighborhood of
$p$ on the transversal section $T$. Fix coordinates
$(z_1,\dots,z_{n-1})$ on $T$. Apply Lemma \ref{new_foliation} to
the curve $\gamma$, the point $p$ and $\Phi_{\lambda}=z+\lambda
a$. Assume that points $q_1\in \gamma_1$, $q_2\in \gamma_2$.

Outside of the flow-box $M_1(\lambda)=\pi_{\lambda}(M_1)$,
$M_2(\lambda)=\pi_{\lambda}(M_2)$.

In a neighborhood of the point $p$:
$$T\cap M_2(\lambda)=\pi^z_{\lambda}(m_2)$$
$$T\cap M_1(\lambda)=\Phi_{\lambda}\circ
\pi^z_{\lambda}(m_1).$$
Therefore, by Sard's Theorem, for almost all $a$ they intersect
transversally.
\end{proof}

\section{Construction of a global eliminating family.}\label{global}

In this section we give the geometric conditions for degenerate
objects to be holomorphically convex and show how to pass from a
local foliation to a global one.

\subsection{Approximation Theory.}

Working in the category of smooth vector fields one can eliminate
a non-transversality by perturbing the vector field only in a
neighborhood of a non-transversality. In the holomorphic category
there are no local perturbations allowed. However, approximation
theory gives a way to work locally. In some cases you can perturb
the local picture and then approximate your perturbation by a
global one. In particular, for a holomorphic vector bundle on a
Stein manifold holomorphic sections over a neighborhood of a
holomorphically convex set can be approximated by global
holomorphic sections. This follows from two theorems formulated
below.

\begin{te}\label{approximation1} {\bf (\cite{Ho}, 5.6.2)}  Let $X$ be a Stein manifold and
$\phi$ a strictly plurisubharmonic function in $X$ such that
$K_{c}=\{z \colon z \in X, \phi(z) \leq c\} \Subset X$ for every
real number $c.$ Let $B$ be an analytic vector bundle over $X.$
Every analytic section of $B$ over a neighborhood of $K_c$ can
then be uniformly approximated on $K_c$ by global analytic
sections of $B.$
\end{te}

\begin{te} {\bf (\cite{Ho}, 5.1.6)}\label{approximation2} Let $X$ be a Stein manifold,
$K$ a compact subset of $X$ and $U$ is an open neighborhood of
holomorphic hull of $K.$ Then there exists a function $\phi \in
C^{\infty}(X)$ such that
\begin{enumerate}
\item $\phi $ is strictly plurisubharmonic,
\item $\phi < 0$ in $K$ but $\phi>0$ in $X \backslash U,$
\item $\{ z \colon z \in X, \phi(z) <c \} \Subset X$ for
every $c \in \Bb R.$
\end{enumerate}
\end{te}

\begin{te}\label{te:global_family} Let $\gamma$ be a holomorphically convex degenerate object
of a foliation ${\cal F}_0$. Then there exists a holomorphic
family ${\cal F}_{\lambda}$ of foliations on $X$, that remove
$\gamma$.
\end{te}

\begin{proof} Let $s_{\lambda}$ be local sections that determine
local foliations that eliminate $\gamma$. Let $\lambda_0\in
\Lambda$ be a parameter that does not belong to exceptional real
curve. By Theorems \ref{approximation1} and \ref{approximation2}
there exists a global section $S_{\lambda_0}$ that is
$\epsilon$-close to $s_{\lambda_0}$ on $U'$, where $\gamma\Subset
U'\Subset U$. Therefore, family of foliations determined by
$S_{\lambda}=S_0+\lambda(S_{\lambda_0}-S_0)$ eliminate the
degenerate object.
\end{proof}

\subsection{Holomorphic convexity of a curve.}

We recall the definition of a holomorphic hull and gave examples
of holomorphic hulls of curves in Appendix \ref{Stein_manifolds}.

Consider a collection of $C^1$~-~smooth real curves $\g_1, \dots,
\g_m$ in $\Bb C^N.$ Their holomorphic hull is described by
Stolzenberg's Theorem \cite{St}:

\begin{te}
Let $\g=\g_1\cup\dots\cup\g_m.$ Then $h(\g)\backslash \g$ is a
(possibly empty) one-dimensional analytic subset of $\Bb C^N
\backslash \g.$
\end{te}

\begin{corollary}\label{St-cor} The statement of the theorem is
true if one replaces $\mathbb C^n$ by a Stein manifold.
\end{corollary}

\begin{proof}
There exists a proper embedding of the Stein manifold $X$ to $\Bb
C^N$ for some large enough $N$ (\cite{Ho}, theorem 5.3.9). Let
$h(\g)$ be the holomorphic hull of $\g$ in $\Bb C^N.$ By
Stolzenberg's theorem $h(\g)\backslash \g$ is an analytic subset
in $\Bb C^N\backslash \g.$

Let us show that $h(\g)\subset X.$ $X$ is a maximal spectrum of
functions that are equal to zero on $X$ (\cite{GR}, Theorem VII,
A18). Take a function $f$ such that $f(X)=0.$ Then $f(h(\g))=0$
since $\g\subset X$ and $h(\g)$ is a holomorphic hull of $\g$ in
$\mathbb C^N$. Thus, $h(\g) \subset X.$

It remains to show that $h(\g)=h_X(\g).$ Any holomorphic function
on $X$ is a restriction of holomorphic function on $\Bb C^N$
(\cite{GR}, Theorem VII, A18). Therefore, $h_X(\g)= h(\g) \cap X.$
Since $h(\g)\subset X,$ $h_X(\g)=h(\g).$
\end{proof}

\subsection{Holomorphic convexity of a degenerate object.}

In this subsection we give the geometric conditions for the
degenerate objects to be holomorphically convex.

\begin{df} We say that a path or a loop is {\it simple} if it does
not have points of self-intersection.
\end{df}

We need to extend the analytic set, given by Stolzenberg's
theorem, to the boundary. In the sequel we need the following
corollary from the Stolzenberg's Theorem.

\begin{corollary}\label{cor:arc} Let $\gamma_1, \dots,\gamma_n$ be piece-wise
smooth curves, such that $\gamma_i\cap \gamma_j$ consists of
finite number of points. Suppose that $h(\gamma)\not \subset
\gamma$. Then there exists an arc $\alpha\subset\g_i$, such that
$\alpha\subset \partial \left(h(\gamma)\backslash\gamma\right)$,
where $\gamma=\cup \gamma_i$.
\end{corollary}

\begin{proof} Let $\pi$ be a projection of $\gamma_1,\dots, \gamma_n$
to a complex line $C$. One can choose $\pi$ so that image of
$\gamma$ has at most finite number of points of transversal
self-intersection. The image of $\gamma$ separate $C$ into several
regions $U_i$.

Below we show that for each of the regions there is a following
dichotomy: either $\pi(h(\gamma))\supset U_i$ or
$\pi(h(\gamma))\cap \mbox{int}(U_i)=\emptyset$:

First, $\pi \left(h(\gamma)\backslash\gamma\right)$ is open.
Therefore, the $\pi(h(\gamma)\backslash \gamma)\cap U_i$ is open.
Second, if $w~\in~\partial~\pi~(~h(~\gamma~)~)$, then since
$h(\gamma)$ is compact, there exists $z\in h(\gamma)$ such that
$\pi(z)=w$. Therefore, $z\in \gamma$. Thus,
$\pi(h(\gamma)\backslash \gamma)\cap U_i$ is either empty or
coincides with $U_i$.

Take a point $w\in \pi(\gamma)$ that is not a point of
self-intersection and belongs to the boundary of $\pi(h(\gamma))$.
The small arc around this point also belongs to $\pi(h(\gamma))$
and does not contain points of self-intersection. The preimage of
this arc is the desired arc on $\gamma$.
\end{proof}

\begin{te}[\cite{Chirka}]\label{the1} Let $M$ be a connected $(2p-1)$-dimensional $C^1$-submanifold
of a complex manifold $\Omega$. Let $A_1$, $A_2$ be irreducible
$p$-dimensional analytic subsets of $\Omega\backslash M$ such that
the closure of each of them contains $M$. Then either $A_1=A_2$ or
$A_1\cup M\cup A_2$ is an analytic subset of $\Omega$.
\end{te}

\begin{lemma}\label{lem:geom} Let $\alpha\subset \gamma$ be a real-analytic arc,
such that $\alpha\subset \partial(h(\gamma)\backslash\gamma)$. Let
$C$ be a holomorphic curve, such that $\alpha\subset C$. Then
there exists a loop $\tilde{\gamma}\subset \gamma$, so that
$\alpha\subset \tilde{\gamma}\subset \gamma\cap C$ and
$\tilde{\gamma}$ is null homologous on $C$.
\end{lemma}

\begin{proof}One can take a
neighborhood $U\subset X$ of the arc $\alpha$, such that
\begin{enumerate}
\item $U\cap \gamma=\alpha$;
\item the connected component of $C\cap U$, that contains $\alpha$, is a submanifold in
$U$;
\item the arc $\alpha$ separates this connected component
into two pieces. Let $\Omega_1,$ $\Omega_2$ be these pieces.
\end{enumerate}
Let $h_1$ denote the connected component of
$h(\gamma)\backslash\gamma$.

Apply Theorem \ref{the1} to the analytic sets $h_1$ and $\Omega_1$
and the arc $\alpha$. The closure of $h_1$ in $U$ contains
$\alpha$. The closure of $\Omega_1$ also contains $\alpha$.
Therefore, either $h_1=\Omega_1$ or $h_1\cup\alpha\cup\Omega_1$ is
an analytic subset of $U$. In the second case $h_1=\Omega_2$.
Thus, $h_1=\Omega_1$ or $h_1=\Omega_2$. If two analytic sets
coincide locally, then they coincide globally. Therefore,
$h_1\subset C$.

By Maximum Modulus Principle, $\partial h(\gamma)\subset \gamma$.
Denote $\tilde{\gamma}=\partial h_1$, then it is a loop and is
null-homologous on $C$.
\end{proof}

\begin{te} Let $\gamma$ be a degenerate object of the foliation
$\cal F$. When $\gamma=\cup \gamma_i,$ then we assume $\gamma_i$
are simple piece-wise real-analytic. Suppose $\gamma$ satisfies
the following additional conditions:
\begin{enumerate}
\item[\it type 1]: Let $\gamma$ be non-homologous to $0$ on the leaf $L$.

\item[\it type 2:]
\begin{enumerate}
\item $\gamma_1$ and $\gamma_2$ have only one common point;
\item $\gamma_1$ and $\gamma_2$ are not null homologous and are not multiples of the same cycle in the homology group of $L$.
\end{enumerate}

\item[\it type 5:] $\gamma$ is not null-homologous on $S$.

\item[\it type 8:] $\gamma_1\subset L$ is non-homologous to $0$ on
$L$; $\gamma_1$ and $\gamma_2$ have only one common point.
\item[\it type 9:] $\gamma_1\subset L_1$, $\gamma_2\subset L_2$ are non-homologous to
zero on $L_1$, $L_2$ correspondingly, $L_1\neq L_2$. Curves
$\gamma_1$ and $\gamma_3$; $\gamma_2$ and $\gamma_4$ have only one
common point.
\end{enumerate}
Then $\gamma$ is holomorphically convex.
\end{te}

\begin{note} If $\gamma$ satisfies the listed above geometric
conditions, then we say that $\gamma$ is a {\it geometric}
degenertate object.
\end{note}
\begin{proof}
Suppose $\gamma$ is not holomorphically convex.
\begin{enumerate}
\item[\it type 1:] Since $\gamma$ is a simple cycle, by Lemma \ref{lem:geom},
$\gamma$ is null-homologous on $L$, which contradicts the
hypothesis.

\item[\it type 2:] The proof is the same as for type $1$.

\item[\it type 3:] $\gamma$ is simply connected, therefore, by Lemma \ref{lem:geom} it
should bound the region on $S$, which contradicts the hypothesis.

\item[\it type 4:] Let $\tilde{S}$ be a surface obtained from $S\cup \{a\}$ by
splitting the local components of $S$ at the point $a$. Let
$\pi:\tilde{S}\to S$ be the corresponding projection.

$\pi^{-1}(\gamma)$ is a simple path on $\tilde{S}$. It does not
bound a region on $\tilde{S}$. Therefore, its image does not bound
a region on $S$. Which contradicts Lemma \ref{lem:geom}.

\item[\it type 5:] The same as for type $1$.

\item[\it type 6:] Let $\alpha$ be an arc, given by Corollary \ref{cor:arc}. Suppose $\alpha\subset \gamma_1$
Let $C$ be a curve, given by \ref{lem:geom}. Then $C$ is either a
saddle connection or $C\subset M_1$.  Saddle connections are
removed. Therefore, by Lemma \ref{lem:geom} $\gamma_1$ bounds a
region on $C$. Which contradict the hypothesis.

\item[\it type 7:] The proof is the same as for type $6$.

\item[\it type 8:] Let $\alpha$ be an arc, given by Corollary
\ref{cor:arc}. Suppose $\alpha\subset \gamma_1$, then by Lemma
\ref{lem:geom}, $\gamma_1$ is null-homologous on $L_1$, which
contradicts the hypothesis. If $\alpha\subset \gamma_3$, then we
proceed by the same reasoning as for type $6$.

\item[\it type 9:] The proof is the same as for type $8$.
\end{enumerate}
\end{proof}

\section{Simultaneous elimination of degeneracies.}\label{simultaneous}

\subsection{Landis-Petrovskii's lemma.}\label{LP_lemma}

The idea is to encode degeneracies by countably many objects. To
give a feeling of the method used, we first prove a version of the
Landis-Petrovskii's Lemma \cite{LP} that we need in the sequel.

\begin{lemma} For a holomorphic $1$-dimensional (singular) foliation $\cal F$ of a
Stein\\
manifold $X$ there exists not more than a countable number of
isolated complex cycles on the leaves of the foliation.
\end{lemma}

\begin{proof} Since the manifold $X$ is Stein, it can be embedded into $\mathbb C^N$.
Take a cycle $\gamma$ on a leaf $L$.

Fix coordinates $(z_1, \dots, z_N)$ in $\mathbb C^N$. Let $C_1,
\dots, C_N$ be the coordinate lines,
$$C_i=\{z_1, \dots,=\hat{z}_i=\dots= z_N=0\}.$$
\noindent Suppose that $L$ does not belong to the hypersurface
$\{z_i=c\}$ for any $c\in\mathbb C$. By perturbing $\gamma$ on the
leaf $L$ one can assume that there exists a small neighborhood
$U\supset\gamma$ so that $\pi_i|_U$ is a biholomorphism to the
image (here $\pi_i:\mathbb C^N\to C_i,\ \pi_i(z)=z_i$ is the
projection). Then one can perturb $\gamma$ inside $U$ so that
$\pi_i(\gamma)$ becomes a piece-wise linear curve with rational
vertices.

\begin{df}
We will say that the cycle $\gamma'$ {\it lies over} the
piece-wise linear curve $g'$ if there exist a representative of
$\gamma'$ and its neighborhood $U',$ such that $U'$ is projected
biholomorphically to its image and the representative is projected
to $g'.$ Note, that any cycle lies over countably many piece-wise
linear curves.
\end{df}

Take one of the vertices of $\pi_i(\gamma)$, say with coordinate
$z_i=c$. The hypersurface $\{z_i=c\}$ intersects $X$ by
$(k-1)$-dimensional variety, such that for any cycle $\gamma'$,
lying over $\pi_i(\gamma)$, it is transversal to the foliation in
a neighborhood of $\gamma'\cap\{z_i=c\}$. The holonomy map along
$\gamma$ is well-define in some neighborhood of the intersection
$\{z_i=c\}\cap \gamma$. The holonomy map does not have any other
fixed points in some smaller neighborhood. Thus, each cycle that
projects to the same piece-wise linear curve gives a neighborhood
on the hyperplane $\{z_i=c\}\subset\mathbb C^N,$ so that two
neighborhoods for two different cycles do not intersect each
other. Therefore, there exists not more than countably many limit
cycles that project to the same curve. Since there are only
countably many curves, there are not more than countably many
limit cycles.

\end{proof}
Landis-Petrovskii's Lemma implies that once all non-isolated
cycles are eliminated, all leaves except for countably many are
homeomorphic to disks.

\subsection{Simultaneous elimination of non-isolated cycles.}
\label{sec:sim_non_iso_cycles} If there are non-isolated cycles on
the leaves of a foliation $\cal F$, then the number of the cycles
is obviously uncountable. However, the strategy described above
can be applied. Our idea is to catch the degenerations by a
countable number of holonomy maps.
\begin{te}\label{te:si_nc} There exists a residual set ${\cal R}_1$ in the space of
$1$-dimensional singular holomorphic foliations, that do not have
geometric degenerate objects of type $1$.
\end{te}

\begin{proof} Since $X$ is Stein, it can be embedded into $\mathbb C^N$.
We can restrict ourselves to the foliations without leaves that
belong to the hypersurfaces $\{z_N=c\},\ c\in\mathbb C$. The set
of such foliations is open and dense. We describe the holonomy
maps that catch all the cycles for all foliations.

We introduce the following notations:
\begin{itemize}
\item $\cal A$ is a countable, everywhere dense subset in the set of
holomorphic foliations;
\item ${\cal G}$ is the set of all closed piecewise-linear curves with rational
vertices on $$\{z_1=\dots=z_{N-1}=0\},$$ \noindent with one marked
vertex.
\item Let $\tau_q=\{z_n=q\}\cap X$, where $q\in \mathbb Q + i\mathbb Q$.

Let ${\cal Q}_q$ be a countable everywhere dense set on $\tau_q$.

Let ${\cal Q}=\bigsqcup{\cal Q}_q$.
\end{itemize}

Let ${\bf z}=\left(z_1,\dots,z_{N-1}\right)$, $u=z_N.$

Consider a 4-tuple $\alpha=({\cal F},g,{\bf z},r) \in ({\cal A},
{\cal G}, {\cal Q}_q,\mathbb Q)$, such that $q$ is the marked
point of $g.$ We require that the holonomy map for the foliation
$\cal F$ in the point ${\bf z}$ along $g$ is well-defined in a
neighborhood of $\bf z$ on the transversal section $\tau_q$ and
has the radius of convergence greater than $r.$ Let
$\Delta_{\alpha}$ be the germ of this holonomy. One can consider
the germ of the holonomy map along the lifting of $g$, starting at
$\bf z$, for foliations close to $\cal F.$ Therefore, we think of
$\Delta_{\alpha}$ as of function of two variables: a foliation
close to $\cal F,$ and a point on the transversal section
$\tau_q$.

Below we fix a specific representative of $\Delta_{\alpha}$. We
use the same notation for the specific representative as for the
germ.

Let $V_{\alpha}$ be the connected component, containing $\cal F,$
of the set of foliations, for which the holonomy map along $g$ in
the point ${\bf z}$ is well-defined and has radius of convergence
greater than $r$. The domain of definition of $\Delta_{\alpha}$ is
$$\{({\cal F}',{\bf z}') |\ {\cal F}'\in V_{\alpha},\ |{\bf z}'-{\bf z}|<r\}.$$
Note, that $V_{\alpha}$ is open.

From this point we consider fixed representatives, rather than
germs.
\begin{lemma}\label{caught_hm} Every complex cycle corresponds to a fixed point of
$\Delta_{\alpha}({\cal F}',\cdot)$ for some $\alpha$ and ${\cal
F}'\in V_{\alpha}.$
\end{lemma}

\begin{proof} Let $\g$ be a complex cycle on a leaf $L$ of a foliation ${\cal
F}$. One can perturb $\g$ on $L$ so that it projects to some $g\in
{\cal G}.$ Let $u(g)$ be one of the vertices of the projection,
and let $\bf z\in\g$ be the preimage of $u(g).$ Consider the
holonomy map along $\g$ in a neighborhood of $\bf z$ in the
transversal section $C=\{u=u(g)\}.$ Take a point ${\bf z}_1 \in
{\cal Q}$ such that $|{\bf z}-{\bf z}_1|<r_{\bf z}({\cal F})/4$
where $r_{\bf z}({\cal F})$ is a radius of convergence of the
holonomy map in the point $\bf z$ along $\g$ for the foliation
${\cal F}$. Note, that $r_{{\bf z}_1}({\cal F})>r_{\bf z}({\cal
F})/2.$ One can take ${\cal F}_1$ close to $\cal F$ so that
$r_{{\bf z}_1}({\cal F}_1)>r_{\bf z}({\cal F})/2.$ Denote by
$\alpha=({\cal F}_1, g,{\bf z}_1,r),$ where $r\in\mathbb Q,r_{\bf
z}({\cal F})/4<r<r_{\bf z}({\cal F})/2.$ Then $r<r_{{\bf
z}_1}({\cal F}_1).$ Also, ${\cal F}\in V_{\alpha},$ because
$r_{{\bf z}_1}({\cal F})>r.$ Since $r>r_{\bf z}({\cal F})/4$, the
point $z$ belongs to the domain of definition of
$\Delta_{\alpha}({\cal F}_1,\cdot)$.
\end{proof}

\begin{lemma}\label{one_germ} Fix $\Delta_{\alpha}$. The set
$D_{\alpha}\subset V_{\alpha}$ of foliations ${\cal F}$ such that
$\Delta_{\alpha}({\cal F},\cdot)$ has a non-hyperbolic fixed
point, so that the corresponding cycle $\gamma$ satisfies
additional conditions: \begin{enumerate}
\item $\gamma$ is simple,
\item $\gamma$ is null homologous on the leaf;
\end{enumerate}
\noindent is closed and nowhere dense in $V_{\alpha}$.
\end{lemma}

\begin{proof}
We prove that by a finite number of steps, we can perturb the
foliation $\cal F$ so that $\Delta_{\alpha}(\widetilde{\cal
F},\cdot)$ has isolated fixed points only in the domain of
definition discussed above. Assume that $A$ is the set of fixed
points of $\Delta_{\alpha}(\cal F,\cdot)$. Let $A$ be
$k$-dimensional. As we show in the appendix, one can associate
multiplicity $m(A)$ to the analytic set $A$. Take a point $z$ that
is a generic point of a $k$-dimensional stratum $A_i$. By Theorem
\ref{te:global_family}, there exists a neighborhood of $z$ and a
foliation $\widetilde{\cal F}$, arbitrary close to ${\cal F}$,
such that the holonomy map of $\widetilde{\cal F}$ along $\gamma$
has isolated fixed points only in this neighborhood.

This perturbation destroys the component $A_i.$ Therefore, it
either decreases the dimension of $A$, or it decreases the
multiplicity $m(A)$ (see Lemma \ref{multiplicity}).

Therefore, after a finite number of steps, only isolated cycles
are left. By the Theorem \ref{te:global_family}, they can be
turned into hyperbolic by a finite number of steps as well.
\end{proof}

\begin{corollary}
The complement of $D_{\alpha}$ in the set of all foliations
contains an open every where dense set.
\end{corollary}
The residual set is obtained by intersecting open everywhere dense
sets from the corollary above.
\end{proof}

\subsection{Simultaneous splitting of cycles to different
leaves.}\label{simultaneous_pair_cycle}

\begin{te}\label{te:si_pc} There exists a residual set in the space of singular
holomorphic $1$-dimensional foliations that do not have geometric
degenerate objects of type $2$.
\end{te}
\begin{proof} The construction is similar to Section
\ref{sec:sim_non_iso_cycles}. The difference is that one needs to
consider pairs of holonomy maps and the analytic condition is that
they do not have a common fixed point.
\end{proof}

\subsection{Simultaneous elimination of separatrices and non-transversal intersections of invariant manifolds}

\begin{te}\label{te:si_im} There exists a residual set in the space of singular
holomorphic $1$-dimensional foliations that do not have geometric
degenerate objects of types $3-9$.
\end{te}

\begin{proof} We outline the proof for strongly invariant
manifolds of different singular points. For other types of
degenerate objects the proof goes along the same lines.

Since $X$ is a Stein manifold, it can be embedded into $\mathbb
C^N$.

We fix the countable set of data $\alpha=({\cal F}, a_1, M_1,
a_2,M_2,g,z_1,r)$.
\begin{itemize}
\item ${\cal F}\in {\cal A}$, where $\cal A$ is a countable
every-where dense set of foliations;
\end{itemize}

Foliations with hyperbolic singular points only form a residual
set \cite{Chap}. Therefore, we can assume that all singular point
for all the foliations ${\cal F}\in {\cal A}$ are hyperbolic.
\begin{itemize}
\item $a_1,a_2$ are hyperbolic singular points of ${\cal F}$;
\item $M_1, M_2$ are strongly invariant manifolds of $a_1$ and
$a_2$ correspondingly;
\end{itemize}

We associate the maximal radius $r_i$ to the singular point $a_i$.

\begin{df} The radius $r_i$ is the maximal radius, such
that $M_i$ is transversal to $\partial U_{r}(a_i)$ for all
$r<r_i$.
\end{df}

Not that maximal radius is a lower semicontinuous function  on the
space of foliations.

Let $\pi:X\to C$ be the projection to $C=\{z_1=\dots=z_{N-1}=0\}$,
$\pi(x_1,\dots, x_N)=x_N$.

\begin{itemize}
\item $g\subset C$ is a piecewise linear curve with rational
vertices. Let $u_1$, $u_2$ be the starting and the ending points
of $g$ correspondingly. We require that $u_1\in
\pi(U_{r_1}(a_1))$, $u_2\in \pi(U_{r_2}(a_2))$;
\item $z_1\in {\cal Q}_q$, where ${\cal Q}_q$ is an every where dense set on the
transversal section $\tau_1=\{z_n=u_1=q\}\cap X$ in
$U_{r_1}(a_1)$;
\end{itemize}

We require that there is a well-defined lift of $g$ to the leaf
$L$ of the foliation $\cal F$, that starts from a point $z_1$. The
lift is denoted by $\gamma$. Let $z_2$ be the lift $u_2$. We
require that $z_2\in U_{r_2}(a_2)$

Let $\tau_2=\{z_N=u_2\}\cap X.$

There is a well-defined germ $\Delta:\tau_1\to \tau_2$ of the
holonomy map along $\gamma$ in the point $z_1$.

As before, we think of $\Delta$ as a function of two variables: a
foliation ${\cal G}$, close to $\cal F$, and a point on the
transversal section $\tau_1$.
\begin{itemize}
\item $r\in {\cal Q}_{+}$. We require that
\begin{enumerate}
\item $r$ is less than radius of convergence of $\Delta$.
\item The disk $D_r(z_1)$ on the transversal section $\tau_1$ of
the radius $r_1$ with the center $z_1$ is compactly contained in
$U_{r_1}(a_1)$.
\item $\Delta(D_r(z_1))$ is compactly contained in $U_{r_2}(a_2)$.
\end{enumerate}
\end{itemize}

We fix a representative $\Delta_{\alpha}$ of $\Delta$. Below we
describe the neighborhood $U_{\alpha}$ of $\cal F$. $\cal G$
belongs to $U_{\alpha}$ if

\begin{enumerate}
\item there is a holomorphic family of foliations ${\cal F}_{\lambda}$,
so that ${\cal F}_0=\cal F$, ${\cal F}_1={\cal G}$; for all
$\lambda\in D_1$ there are unique hyperbolic singular points
$a^{\lambda}_1\in U_{r_1/2}(a_1)$ and $a^{\lambda}_2\in
U_{r_2/2}(a_2)$ of the foliation ${\cal F}_{\lambda}$;

Let $a'_1$, $a'_2$ be singular points of $\cal G$, obtained via
holomorphic continuation. Let $M'_1$, $M'_2$ be the corresponding
strongly invariant manifolds. Let $r'_1$, $r'_2$ be the maximal
radii for $(a'_1, M'_1)$, $(a'_2, M'_2)$.
\item $z_1\in U_{r'_1}(a'_1)$, $z'_2\in U_{r'_2}(a'_2)$, where
$z'_2$ is the lift of $u_2$ along $g$ for $\cal G$.
\item $D_r(z_1)$ is compactly contained in $U_{r'_1}(a'_1)$.
\item $\Delta({\cal G}, D_r(z_1))$ is compactly contained in
$U_{r'_2}(a'_2)$.
\end{enumerate}

The domain of definition of $\Delta_{\alpha}$ is $U_{\alpha}\times
D_r(z_1)$.

\begin{lemma} For any $\alpha$, the set $D_{\alpha}\subset U_{\alpha}$ of foliations ${\cal G}\subset U_{\alpha}$,
for which there exists a leaf $L$ such that
\begin{enumerate}
\item the lift of $u_1$ to $L$ is in $U_{r_1}(a_1)$, the lift of $u_2$
to $L$ is in $U_{r_2}(a_2)$;
\item the lift of  $g$ belongs to the strongly invariant manifold $M'_1$ of the
singular point $a'_1$ of $\cal G$ ($a'_1$ is a holomorphic
continuation of $a_1$);
\item the lift of $u_2$ belongs to the strongly invariant manifold
$M'_2$ of a singular point $a'_2$ ($a'_2$ is a holomorphic
continuation of $a_2$);
\item the lift of $u_2$ is a point of a non-transversal
intersection of $M'_1$ and $M'_2$.
\end{enumerate}
\noindent is a closed and nowhere dense set.
\end{lemma}

\begin{proof} The proof follows from the local Theorem
\ref{te:global_family} in the same way as in Lemma \ref{one_germ}.
\end{proof}

The desired residual set is obtained by intersecting the open
everywhere dense sets from the above Corollary.
\end{proof}

\subsection{Proofs of main theorems.}

\begin{te}\label{te:1} A foliation ${\cal F}$, that does not have
geometric degenerate objects of types $1-5$ satisfies Theorem \ref{main1}.
\end{te}

\begin{proof} If a leaf $L$ is not contractible, then there exists a simple loop $\gamma\subset
L$, non-homologous to zero on $L$. Foliation ${\cal F}$ does not have geometric non-isolated cycles. Hence,
if $L$ is a non-contractible leaf of a foliation $\cal F$, then there is a geometric isolated cycle $\gamma\subset L$.
By Landis-Petrovskii's Lemma (Section \ref{LP_lemma}), there are at most countably many isolated cycles. Thus, all leaves,
except for countably many, are contractible.

If $H^1(L,\mathbb Z)\neq 0, \mathbb Z$, then there exist a pair of cycles $\gamma_1,\gamma_2\subset L$, that
satisfy geometric conditions. Since foliation $\cal F$ does not have a pair of geometric cycles, all
non-separatrix leaves $L$ are either contractible or $H^1(L, \mathbb Z)=\mathbb Z$.

Since foliation $\cal F$ does not have geometric degenerate objects of types 3-5, one can show the same way, that separatrix
leaves are topological cylinders.
\end{proof}

Theorem \ref{main1} is an immediate corollary of Theorems
\ref{te:1}, \ref{te:si_nc}, \ref{te:si_pc}, \ref{te:si_im}.

\begin{te}\label{te:2} If a foliation ${\cal F}$ does not have
geometric degenerate objects of types $1-6$, $8-9$, then it is
complex Kupka-Smale.
\end{te}

\begin{proof} By Theorem \ref{te:1} all leaves of foliation $\cal
F$ are either contractible or cylinders. Since the foliation does
not have geometric non-hyperbolic cycles, all cycles are
hyperbolic.

Suppose there is a non-transversal intersection of invariant
manifolds $M_1, M_2$. Let $p$ be a point of non-transversal
intersection. Let $L$ be a leaf of foliation, such that $p\in L$.
Since $L\subset M_1$, there is a path $\gamma_1\subset L$ that
connects $p$ with a point $q\in M_1^{loc}$, one can assume that
$\gamma_1$ is simple and piece-wise real analytic. The same way,
there is a leaf-wise $\gamma_2$ path from $p$ to $M_2^{loc}$. Thus
we constructed a geometric degenerate object of type $6-9$, which
contradicts the hypothesis.
\end{proof}

Thus Theorem \ref{main2} is an immediate consequence of Theorem
\ref{main1} and Theorem \ref{te:si_im}.

The same way one shows that Theorem \ref{te:homoclinic_int} is a corollary from Theorem \ref{te:si_im}.

\section{Appendix}

\subsection{Stein manifolds.}\label{Stein_manifolds}

In this subsection we state the well-known facts about Stein
manifolds. For the proofs and further discussion, consult
\cite{Ho}.

Whitney Embedding Theorem states that any smooth $m$-dimensional
manifold can be smoothly embedded into Euclidean $2m$-space. For
complex holomorphic manifolds the situation is different. There
are complex manifolds that cannot be holomorphically embedded as
submanifolds to $\mathbb C^n$. Moreover, there are ones that do
not admit any global holomorphic functions, except for constants.

By Maximum Modulus Theorem and Liouville's Theorem compact
manifolds do not admit any nonconstant global holomorphic
functions.

Informally speaking, Stein manifolds are the ones which do have an
ample supply of holomorphic functions.

We start our discussion of Stein manifolds with the definition of
the holomorphic hull. This notion plays an important role in the
theory.

\begin{df} Let $K$ be a compact subset of a complex manifold $X$,
the {\it ${\cal O}(X)$-hull} of $K$ is the set
$$h_X(K)=\{u:\ |f(u)|\leq \max \{f(x)| x\in K\}\ \mbox{for all}\ f\in {\cal
O}(X)\},$$
\noindent where ${\cal O}(X)$ are holomorphic functions
on $X$.
\end{df}

\begin{note} We also call ${\cal O}(X)$-hull, the holomorphic hull, when it is
clear from the context what the ambient manifold is. The notation
$h(K)$ is used in that case.
\end{note}

\begin{note} The holomorphic hull is a reasonable notion, only if
the manifold has an ample supply of holomorphic functions. For
instance, it is an important notion for the compact subsets of
$\mathbb C^n$.
\end{note}

\begin{example} The holomorphic hull of the curve $\{|z|=1\}\subset \mathbb
C$ is $\{|z|\leq 1\}$, i.e. the curve together with interior in
$\mathbb C$.
\end{example}

\begin{figure}[h!]\centering
\psfrag{|z|=1}{$|z|=1$} \psfrag{formula1}{$h(\{|z|=1\})=\{|z|\leq
1\}$}
\includegraphics[height=2.5cm]{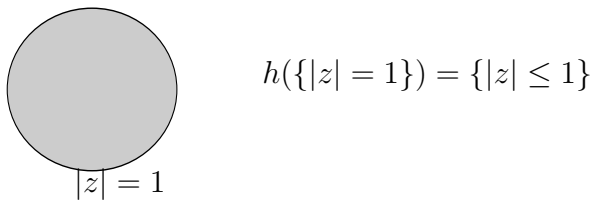}
\caption{The holomorphic hull of the curve $\{|z|=1\}$ in $\mathbb
C$}
\end{figure}

\begin{proof} By Maximum Modulus Principle, the points $z$ such that $\{|z|\leq 1\}$
belong to the holomorphic hull.

Take a point $z_0$ so that $|z_0|>1$. By considering the global
holomorphic function $z$ we see that this point does not belong to
the holomorphic hull.
\end{proof}

\begin{example} Consider a curve $\gamma$
$$\{(z,w)\in \mathbb C^2|\ |z|=1, z=\bar{w}\}.$$
Then $h_{\mathbb C^2}(\gamma)=\gamma$.
\end{example}

\begin{proof} The function $f(z,w)=zw-1$ is equal to zero on
$\gamma$. Therefore, $h(\gamma)\subset \{f=0\}$. Take a point
$(z_0,w_0)\in \mathbb C^2$.

\begin{itemize}
\item If $|z_0|>1$. Then
$$|z_0|>\max\{|z|:\ (z,w) \in \gamma\}$$
Function $z$ is a global holomorphic function. Therefore, the
point $(z_0,w_0)$ does not belong to $h(\gamma)$.
\item If $|z_0|<1,$ then $|w_0|>1$.
$$|w_0|>\max\{|w|:\ (z,w) \in \gamma\}$$
Therefore, the point $(z_0,w_0)$ does not belong to $h(\gamma)$.
\item If $|z_0|=1$, then $z_0=\bar{w}_0$. So $(z_0,w_0)\in \gamma$.
\end{itemize}
\noindent Thus, $h(\gamma)=\gamma$
\end{proof}

\begin{df} Complex analytic manifold $X$ of dimension $n$ is said to be {\it a Stein
manifold} if

\begin{enumerate}
\item for every compact set $K$ its holomorphic hull $h(K)$ is also compact;
\item If $z_1$ and $z_2$ are two different points of $X$, then $f(z_1)\neq f(z_2)$ for some $f\in {\cal O}(X)$;
\item For every $z\in X$, one can find  $n$ functions $f_1,\dots, f_n\in {\cal O}(X)$
which form a coordinate system at $z$.
\end{enumerate}
\end{df}

\begin{fact} $\mathbb C^n$ is a Stein manifold.
\end{fact}

\begin{fact} Every closed submanifold of a Stein manifold is a Stein
manifold.
\end{fact}

In fact there is the Embedding Theorem for Stein manifolds.

\begin{fact} Every Stein manifold can be holomorphically embedded as a closed submanifold
into $\mathbb C^N$.
\end{fact}

Below we give one more equivalent definition of a Stein manifold
in terms of a plurisubharmonic function, that is often used in
practice.

\begin{df} A function $\phi$ defined in an open set $\Omega\subset \mathbb
C^n$ with values in $[-\infty,+\infty)$ is {\it plurisubharmonic}
if
\begin{enumerate}
\item it is semicontinuous from above.
\item For an arbitrary $z$ and $w\in \mathbb C^n$, the function $\tau\to \phi(z+\tau
w)$ is subharmonic in the part of $\mathbb C$ where it is defined.
\end{enumerate}
\end{df}

\begin{fact} A function $\phi\in C^2(\Omega)$ is plurisubharmonic if
and only if

\begin{equation}\label{eq:psh}
\sum_{j,k=1}^n\partial^2 \phi(z)/\partial z_j\partial \bar{z}_k
w_j \bar{w}_k\geq 0,
\end{equation}

\noindent where $z\in \Omega$, $w\in \mathbb C^n$.
\end{fact}

\begin{df} Function $\phi$ is strictly plurisubharmonic if the inequality
(\ref{eq:psh}) is strict.
\end{df}

The notion of plurisubharmonicity does not depend on the choice of
holomorphic coordinates. Therefore, it is well defined on all
complex manifolds.

\begin{fact} A complex manifold $X$ is a Stein manifold if and
only if there exists a strictly plurisubharmonic function $\phi\in
C^{\infty}(X)$ such that
$$\Omega_c=\{z|\ z\in X,\phi(z)<c\}\Subset X$$
for any real number $c$. The sets $\hat{\Omega}_c$ are the ${\cal
O}(X)$-convex.
\end{fact}

\subsection{Complex foliations}

Definitions \ref{def1}-\ref{def_last} are from \cite{IlYa}. They
are scatted through out the text, so we provide them here for the
convenience of the reader. Definition \ref{def_connection},
\ref{def_sim} can be found in \cite{V06},\cite{Chaperon}
correspondingly.

\begin{df}\label{def1} Let $\cal F$ be a foliation of a complex manifold $X.$
Let $\gamma:[0,1] \to X$ be a path on X. Let $T_0$ and $T_1$ be
two transversal sections to $\cal F,$ passing through $\gamma(0)$
and $\gamma(1)$ respectively. Then for any initial point $x\in
T_0$, close to $\gamma(0)$, leaf-wise curves, starting from $x$,
and staying close to $\gamma$, and arriving to $T_1,$ arrive at a
well defined point $\Delta_{\gamma}(x).$ Thus, we obtain a map
$\Delta_{\gamma}(x),$ which we call the {\it holonomy map}. If
$\gamma:[0,s]\to X$ is a closed curve, and $T$ is a transversal
section to $\cal F$, passing through $\gamma(0)$. The map
$\Delta_{\gamma}:T\to T$ is called the {\it holonomy map} as well.
\end{df}
\begin{figure}[h!]\centering
\psfrag{gamma}{$\gamma$}\psfrag{T}{$T$}
\psfrag{Delta(z)}{$\Delta(z)$}\psfrag{z}{$z$}
\includegraphics[height=5cm]{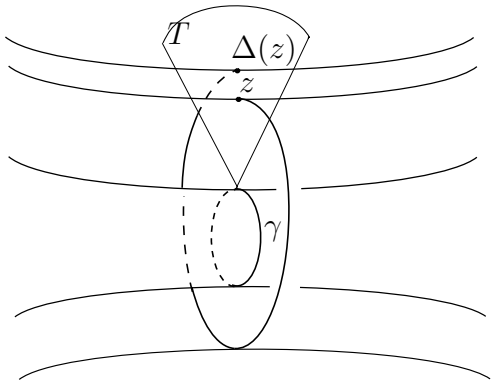}
\caption{A holonomy map}
\end{figure}
\begin{df} {\it A complex cycle} is a nontrivial free homotopy class of
loops on a leaf of a foliation. It is called {\it isolated} if it
corresponds to an isolated fixed point of its holonomy map.
It is {\it hyperbolic} if its holonomy map is hyperbolic, i.e. its
linearization is non-degenerate, and the eigenvalues of the
linearization do not belong to the unit circle.\end{df}

\begin{df} Let $\gamma$ be a hyperbolic cycle. The holonomy map
$\Delta_{\gamma}$ has stable and unstable manifolds $m_1^{loc}$,
$m_2^{loc}$. The union of leaves that pass through $m_1^{loc}$ and
$m_2^{loc}$ are called stable and unstable manifolds of $\gamma$.
\end{df}

\begin{df}
A singular point is called {\it complex hyperbolic} if it is
non-degenerate and the ratio of any two eigenvalues is not real.
\end{df}

In this paper we work only with complex hyperbolic singular
points. So we reserve the word "hyperbolic" to complex
hyperbolicity.

Note that complex hyperbolicity plays a similar role for the
theory of complex vector fields as hyperbolicity for the theory of
real vector fields. In particular, if the point is complex
hyperbolic, then the phase portrait of the vector field in the
neighborhood of a singularity is homeomorphic to the phase
portrait of its linearization \cite{Chaperon}. See (\cite{IlYa},
Section 29) for thorough consideration of properties of complex
hyperbolic points.

\begin{df}\label{def_last} A {\it complex separatrix} of a singular holomorphic
foliation $\cal F$ at a singular point $a\in \Sigma(\cal F)$ is a
local leaf $L\subset (U, a)\backslash \Sigma,$ whose closure
$L\cup a$ is a germ of an analytic curve.
\end{df}

\begin{df}\label{def_connection}
A {\it saddle connection} is a common separatrix of two singular
points. See Fig.\ref{fig:saddle_connection}
\end{df}
\begin{df}\label{def_sim} Suppose $a$ is a hyperbolic singular point of the foliation $\cal F$.
Let $\lambda_1,\dots,\lambda_n$ be the eigenvalues of $a$. Let $l$
be a line passing through the origin in $\mathbb C$. Let
$\lambda=(\lambda_{i_1},\dots,\lambda_{i_k})$ be the eigenvalues
of $a$ that lie on one side of the line $l$. Let
$\alpha_{\lambda}$ be a subspace spanned by the eigenspaces of all
elements of $\lambda$. The local {\it strongly invariant manifold}
$M^{loc}_{\lambda}$ is a manifold tangent to $\alpha_{\lambda}$.
The global strongly invariant manifold $M_{\lambda}$ is obtained
by taking the union of leaves that belong to the local strongly
invariant manifold.
\end{df}

\begin{figure}[h!]\centering

\psfrag{lambda1}{$\lambda_1$}
\psfrag{lambda2}{$\lambda_2$}\psfrag{lambda3}{$\lambda_3$}\psfrag{lambda4}{$\lambda_4$}\psfrag{formula1}{$\lambda=(\lambda_1,\lambda_2,\lambda_4)$}
\psfrag{formula2}{$\lambda=(\lambda_3)$}\psfrag{l}{$l$}
\includegraphics[height=5cm]{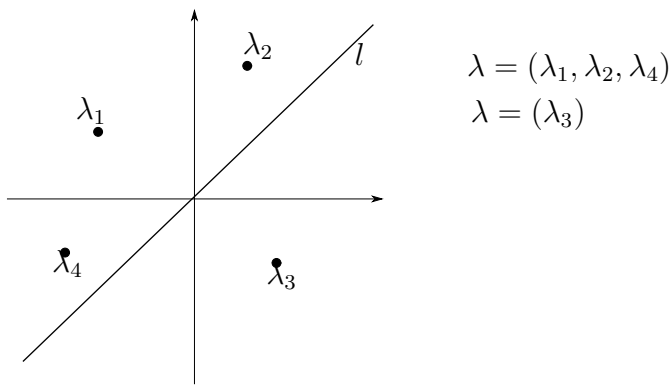}
\caption{A line, separating eigenvalues.}
\end{figure}

Strongly invariant manifolds exist \cite{IlYa}. The proof can be
easily modified to show that they depend holomorphically on a
vector field (on a foliation).

Suppose that $v$ is a vector field that determines a foliation
locally. Strongly invariant manifolds are stable and unstable
manifolds of the time-one map $\Phi^1_{cv}$ of the vector field
$cv$, where $c\in \mathbb C^*$ is taken so that $l$ becomes the
imaginary axis. If one considers the real flow of the vector field
$cv$, then locally strongly invariant manifolds coincide with
stable and unstable manifolds \cite{Chaperon}.

\subsection{Holomorphic vector bundle associated to a foliation}
Take a $1$-dimensional singular holomorphic foliation $\cal F$ of
a Stein manifold $M.$ One can naturally associate a linear bundle
$B_{\cal F}$ to $\cal F.$

Notice that a 1-dimensional holomorphic foliation with singular
locus of codimension $2$ is locally determined by a holomorphic
vector field \cite{IlYa}. Consider a covering of a Stein manifold
by open contractible sets $U_i$. On each set $U_i$ the foliation
is determined by a holomorphic vector field $v_i$. For a pair of
intersecting sets $U_i$ and $U_j$ define a function
$g_{ij}=v_i/v_j.$ This function is well-defined on $\left(U_i\cap
U_j\right)\backslash \{v_j=0\}$. The set $\{v_j=0\}$ has
codimension $2$. Therefore, $g_{ij}$ can be extended to $U_i\cap
U_j$.

The same way $g_{ji}=v_j/v_i$ can be extended to a well-defined
function on $U_i\cap U_j$.
$$g_{ij}g_{ji}=1 \quad \Rightarrow\quad g_{ij}\left|_{U_{i}\cap U_j}\neq 0\right.$$
The set of functions $\{g_{ij}\}$ form a $2$-cocycle, therefore,
they define a linear bundle.
\begin{lemma} $1$-dimensional singular holomorphic foliation $\cal F$ of a Stein manifold $X$ is
determined by a global section of the vector bundle $TX \otimes
B_{\cal F}.$
\end{lemma}

\begin{proof} Lemma follows from the construction of
$B_{\cal F}$.
\end{proof}

If $H_2(X,\mathbb Z)=0$, then each foliation on $X$ is determined
by a global vector field. In particular, this holds for foliations
on $\mathbb C^n$.

\subsection{Topology of uniform convergence on compact
non-singular sets}\label{topology}

The description of topology on the space of foliations in $\mathbb
C^n$ is given for example in \cite{GKK}. Let $X$ be a Stein
manifold. We fix its compact exhaustion:
$$K_1\Subset\dots\Subset K_n\dots\Subset X,$$
\noindent where $K_1,\dots ,K_n$ are compact subsets of $X$,
closures of open connected subsets of $X$;
$$\cup_n K_n=X.$$
Let $d_1$ be a metric on $X$ and $d_2$ be a metric on the
projectivization of its tangent bundle $PTX$. A basis of
neighborhoods of the foliation ${\cal F}$ is formed by
$$U_{n,\epsilon,\delta}=\left\{\parbox{12.5cm}{${\cal G}|\ {\cal G}$ is nonsingular in $K_{\epsilon,
n}=K_n\backslash U_{\epsilon}(\Sigma({\cal F}))$ and the tangent
directions to the foliations ${\cal F}$ and ${\cal G}$ are
$\epsilon$-close on $K_{\epsilon, n}$}\right\}.$$ Note that the
obtained topology does not depend on the choice of compact
exhaustion and the choice of metrics $d_1$ and $d_2$. The set of
foliations of $X$ has countably many connected components,
parametrized by Chern classes of the linear bundles, associated to
the foliations.

The set of sections of $TX\otimes B_{\cal F}$ is equipped with the
topology of uniform convergence on compact sets. (See  for the
description of the linear bundle $B_{\cal F}$.) The map from the
space of sections to the space of foliations is continuous.

\subsection{Multiplicity}
We consider analytic subsets $A$ of a polydisk $\bar{D}^n$, i.e.
we assume that $A$ is an analytic subset of some neighborhood of
$D^n$. Suppose that $A$ is given by a system of $n$ equations
$$
f_1=\dots =f_n=0
$$
Assume that $A$ is $k$-dimensional. We want to define the
multiplicity of $A$ which does not increase under perturbations.

\begin{lemma} There are only finitely many strata of maximal dimension.
\end{lemma}

\begin{proof} The number of strata is locally finite
\cite{Chirka}. Since $A$ is an analytic subset of $\bar{D}^n$, it
is globally finite.
\end{proof}

Let $A_1, \dots, A_m$ be the strata of maximal dimension.

Take a smooth point $z\in A_i$. Consider a transversal section $T$
to $A_i$ at the point $z$. Let $\tilde{f}_1, \dots, \tilde{f}_n$
be the restriction of $f_1, \dots, f_n$ to $T$. The point $z$ is
an isolated solution of the system
$$\tilde{f}_1=\dots=\tilde{f}_n=0.$$

\begin{df} Let $z$ be an isolated point of a system of equations
$$\tilde{f}_1=\dots=\tilde{f}_n=0,$$

\noindent defined in $n-k$-dimensional polydisk $D^{n-k}$. The
multiplicity $m(z)$ of a point $z$ is the dimension of
$$ {\cal O}_{D^{n-k},z}/<\tilde{f}_1,\dots,\tilde{f}_n>,$$
\noindent where ${\cal O}_{D^{n-k},z}$ is the local ring of $z\in
D^{n-k}$, i.e. functions, regular in a neighborhood of $z\in
D^{n-k}$; $<\tilde{f}_1,\dots, \tilde{f}_n>$ is the ideal in
${\cal O}_{D^{n-k},z}$ generated by $\tilde{f}_1,\dots,
\tilde{f}_n$.
\end{df}

\begin{lemma}\label{local_multiplicity} The multiplicity does not increase under perturbations, i.e. if $z'_1,\dots, z'_m$ are isolated solutions of
a perturbed system in a neighborhood of a point $z$, then
$$\sum_{i=1}^{m}m(z'_i)\leq m(z).$$
\end{lemma}

\begin{proof} In (\cite{AGV},Chapter 2,5.7) it is proved for $k=0$. In general
case the proof goes the same way.
\end{proof}

\begin{df} The multiplicity of $z\in A_i$ is the multiplicity the
point $z$ as an isolated solution of
$\tilde{f}_1=\dots=\tilde{f}_n=0$.
\end{df}

It is easy to see that multiplicity does not depend on the choice
of a generic point and a transversal section $T$.

\begin{df} The multiplicity of a stratum $A_i$ is the multiplicity of a
generic point. The multiplicity of $A$ is the sum of
multiplicities of $A_i$.
\end{df}

\begin{lemma}\label{multiplicity} The multiplicity of $A$ does not increase under perturbations,
i.e. let $A'_1,\dots A'_{m'}$ be strata of a perturbed system,
then
$$\sum_{i=1}^{m'}m(A'_i)\leq m(A).$$
\end{lemma}

\begin{proof}
Let $T_1,\dots, T_m$ be transversal sections to $A_i$'s at generic
points. Every $A_i'$ intersect at least one of the sections
$T_1,\dots, T_m$. One can also assume that $T_i$'s meet $A_i$'s
transversally. On each transversal section the result follows from
the Lemma \ref{local_multiplicity}.

\end{proof}


\newpage
\begin{bibdiv}
\begin{biblist}

\bib{AGV}{book}{
title={Singularities of differentiable maps. Vol.II. The classification of
critical points, caustics and wave fronts},
author= {Arnol'd, V. I.},
author= {Gusein-Zade, S. M.},
author= {Varchenko, A. N.},
subtitle={Monographs in Mathematics},
volume={82},
publisher={Birkhauser Boston, Inc.},
address={Boston, MA},
date={1985}}

\bib{Buzzard}{article}{
author={Buzzard, G.T.},
title={Kupka-Smale Theorem for Automorphisms
of $\mathbb C^n$},
journal={Duke Math. J.},
volume= {93},
pages={487-503},
date={1998}}

\bib{Chaperon}{article}{
author={Chaperon, Marc},
title={$C^k$-conjugacy of holomorphic flows near a singularity.}
journal={Inst. Haute \'{E}tudes Sci. Publ. Math.}
date={1986},
volume={64},
pages={143-183}}

\bib{Chap}{article}{
author={Chaperon, Marc}
title={Generic complex flows.}
journal={Complex Geometry II: Contemporary Aspects of Mathematics and
Physics},
publisher={Hermann},
address={Paris},
date={2004},
pages={71-79}}

\bib{Chirka}{book}{
author={Chirka, E.M.},
title={Complex analytic sets},
publisher={Kluwer},
address={Dordrecht},
date={1989}}

\bib{CGM}{article}{
author={Candel, A.},
author={Gomez-Mont, X.},
title={Uniformization of the leaves of a rational vector field},
journal={Annales de l'Institute Fourier},
volume={45(4)},
date={1995},
pages={1123-1133}}

\bib{TF}{article}{
author={Firsova, T.S.},
title={Topology of analytic foliations in $\Bb C^2.$ Kupka-Smale property},
journal={Proceedings of the Steklov Institute of Mathematics},
date={2006},
volume={254},
pages={152-168}}

\bib{GR}{book}{
author= {Ganning, R.},
author= {Rossi, H.},
title={Analytic functions of several complex variables},
publisher={Prentice-hall, Inc.},
address={Englewood Cliffs, N.J.},
date={1965}}

\bib{G94}{article}{
author={Glutsyuk, A.},
title={Hyperbolicity of phase curves of a
general polynomial vector field in $\mathbb C^n$},
journal={Func. Anal. Appl.},
volume={28(2)},
date={1994},
pages={1-11}}

\bib{GK}{article}{
author={Golenishcheva-Kutuzova,T.I.}, title={A generic analytic
foliation in $\mathbb C^2$ has infinitely many cylindrical
leaves}, journal={Proc. Steklov Inst. Math.}, volume={254},
date={2006}, pages={180-183}}

\bib{GKK}{article}{
author={T. Golenishcheva-Kutuzova},
author={V. Kleptsyn},
title={Minimality and ergodicity of a generic foliation of $\mathbb C^2$},
journal={Ergod. Th. \& Dynam. Sys},
date={2008},
volume={28},
pages={1533-1544}}


\bib{Ho}{book}{
author={L. H\"{o}rmander},
title={An Introduction to complex analysis in several variables},
publisher={North Holland},
address={Neitherlands},
edition={third},
date={1990}}




\bib{Il04}{article}{
author={Yu. Ilyashenko},
title={Selected topics in differential equations with real and complex time. Normal forms,
bifurcations and finiteness problems in differential equations},
pages={317-354},
journal={NATO Sci. Ser. II Math. Phys. Chem.},
volume={137},
date={2004}}

\bib{Il08}{article}{
author={Yu. Ilyashenko},
title={Some open problems in real and complex dynamical systems},
journal={Nonlinearity},
volume={21(7)},
date={2008},
pages={101-107}}

\bib{IlYa}{book}{
author={Yu. Ilyashenko},
author={S. Yakovenko},
title={Lectures on analytic differential equations},
series={Graduate Studies in Mathematics},
volume={86},
publisher={Amer. Math. Soc.},
date={2008}}


\bib{LP}{article}{
author={E. Landis}, author={I. Petrovskii}, title={On the number
of limit cycles of the equation
$\frac{dy}{dx}=\frac{P(x,y)}{Q(x,y)},$ where $P$ and $Q$ of second
degree (Russian)}, journal={Mat. sbornik}, volume={37(79), ¹2},
date={1955}}

\bib{LN94}{article}{
author={A.Lins Neto}, title={Simultaneous uniformization for the
leaves of projective foliations by curves}, journal={Bol. Soc.
Brasil. Mat. (N.S.)}, volume={25, no.2}, date={1994},
pages={181-206}}


\bib{Moldavskis}{article}{
author={V. Moldavskis},
title={New generic properties of complex and real dynamical systems},
journal={PhD thesis, Cornell University},
date={2007}}



\bib{Siu}{article}{
author={Yum-Tong Siu},
title={Every Stein subvariety admits a Stein Neighborhood}
journal={Inventiones Math.},
volume={38},
pages={89-100},
date={1976}}

\bib{St}{article}{
author={G. Stolzenberg}, title={Uniform approximation on smooth
curves}, journal={Acta. Math}, date={1966}, volume={115,¹ 3-4},
pages={185-198}}

\bib{V06}{article}{
author={D.S. Volk},
title={The density of separatrix connections in the space of polynomial foliations in
$\Bbb C{\rm P}^2$},
journal={Proc. Steklov Inst. Math.},
date={2006},
volume={ 3(254)},
pages={169-179}}

\bib{Wermer}{article}{
author={J. Wermer}, title={The hull of curve in $\mathbb C^n$},
journal={Annals of Mathematics}, volume={68, 3}, date={1958}}

\end{biblist}
\end{bibdiv}
\end{document}